\documentclass
{amsart}
\usepackage{graphicx}
\newtheorem{thm}{Theorem}[section]
\theoremstyle{definition}
\newtheorem{cor}[thm]{Corollary}
\newtheorem{prop}[thm]{Proposition}

\newtheorem{note}[thm]{Notation and Remark}
\newtheorem{rem}[thm]{Remark}
\newtheorem{ex}[thm]{Example}

\numberwithin{equation}{section}
\begin{document}
\title[Finitely generated coreduced comultiplication modules]
{Finitely generated coreduced comultiplication modules}
\author{F.  Farshadifar}
\address{Department of Mathematics, Farhangian University, Tehran, Iran.}
\email{f.farshadifar@cfu.ac.ir}

\subjclass[2010]{13C13, 13C99.}%
\keywords {finitely generated module, comultiplication module, maximal second submodule, coreduced module.}

\begin{abstract}
This paper deals with some results concerning finitely generated coreduced comultiplication modules over a commutative ring.
\end{abstract}
\maketitle

\section{Introduction}
\noindent
Throughout this paper, $R$ will denote a commutative ring with
identity. Let $M$ be an $R$-module.
$M$ is called a \emph{reduced module} if $rm = 0$ implies that $rM \cap Rm = 0$, where $r\in R$ and $m \in  M$ \cite{MR2050725}.
A proper submodule $N$ of $M$ is said to be \emph{completely irreducible} if $N=\bigcap _
{i \in I}N_i$, where $ \{ N_i \}_{i \in I}$ is a family of
submodules of $M$, implies that $N=N_i$ for some $i \in I$. Every submodule of $M$ is an intersection of completely irreducible submodules of $M$. Thus the intersection
of all completely irreducible submodules of $M$ is zero \cite{FHo06}.
$M$ is said to be \emph{coreduced module} if $(L:_Mr)=M$ implies that $L+(0:_Mr)=M$, where $r \in R$ and $L$ is a completely irreducible submodule of $M$ \cite{MR3755273}. $M$ is said to be a \emph{comultiplication module} if for every submodule $N$ of $M$ there exists an ideal $I$ of $R$ such that $N=(0:_MI)$ \cite{MR3934877}.

Let $M$ be an $R$-module. A non-zero submodule $N$ of $M$ is said to be \emph{second} if for each $a \in R$, the homomorphism $ N \stackrel {a} \rightarrow N$ is either surjective or zero \cite{MR1879449}.  A second submodule $N$ of $M$ is said to be a \emph{maximal second submodule} of $M$, if there does not exist a second submodule $L$ of $M$ such that $N \subset L \subset M$ \cite{MR3073398}.
The set of all maximal second submodules of $M$ will be denoted by $Max^s(M)$. The sum of all maximal second submodules of $M$ contained in a submodule $K$ of $M$ is denote by $\mathfrak{S}_K$.  If $N$ is a submodule of $M$, define
 $V^s(N) = \{S \in Max^s(M) : S \subseteq N\}$.

The purpose of this paper is to obtain some results about finitely generated coreduced comultiplication $R$-modules. In Section 3 of this paper, among the other results, it is shown that if $M$ is a finitely generated faithful coreduced comultiplication $R$-module, then
\begin{itemize}
\item for each $a \in R$,  we have $V^s((0:_MAnn_R(a))=Max^s(M) \setminus V^s((0:_Ma))$ (see Theorem \ref{tt1.11}).
\item $(0:_M\mathfrak{P}_I)=\mathfrak{S}_{(0:_MI)}$ for each ideal $I$ of $R$ (see Theorem \ref{t119.3}).
\item $\mathfrak{S}_{IM}= sec(IM)=IM$ for each ideal $I$ of $R$ (see Theorem \ref{tt0.1}).
\item for each $a \in R$, $Ann_R(aM)M=\mathfrak{S}_{(0:_Ma)}$  (see Corollary \ref{c1.13}).
\item a submodule $S$ of $M$ is a maximal second submodule of $M$
if and only if $S=I^M_{Ann_R(S)}(M)$ (see Corollary \ref{c00.3}).
\end{itemize}

\section{Main results}
One may think that if $S$ is a maximal second submodule of an $R$-module $M$, then $Ann_R(S)$ is a prime ideal minimal over
$Ann_R(M)$. In \cite[Example 5.10]{MR3266519}, it is shown that this is not true in general.
\begin{prop}\label{l0.1}
Let $M$ be a finitely generated comultiplication $R$-module. Then we have the following.
\begin{itemize}
\item [(a)] If $\mathfrak{p}$ is a prime ideal of $R$ and $N$ is a submodule of $M$ with $Ann_R(N) \subseteq \mathfrak{p}$, then $\mathfrak{p}=Ann_R((0:_N\mathfrak{p}))$.
\item [(b)] If $\{\mathfrak{p_i}\}_{i \in I}$ is a set of prime ideals of $R$ with $Ann_R(M) \subseteq \mathfrak{p_i}$, then
$$
(0:_M\cap_{i \in I}\mathfrak{p_i})=\sum_{i \in I}(0:_M\mathfrak{p_i}).
$$
\item [(c)] If $S$ is a maximal second submodule of $N$, then $Ann_R(S)$ is a prime ideal minimal over
$Ann_R(N)$.
\item [(d)] If $S$ is a submodule of $M$ such that $Ann_R(S)$ is a prime ideal minimal over
$Ann_R(M)$, then  $S$ is a maximal second submodule of $M$.
\item [(e)] If $M$ is a coreduced $R$-module, $S$ a maximal second submodule of $M$, and $a \in Ann_R(S)$, then
$$
Ann_{R/Ann_R(M)}(a+Ann_R(M)) \not \subseteq Ann_R(S)/Ann_R(M).
$$
\end{itemize}
\end{prop}
\begin{proof}
(a) Let $\mathfrak{p}$ be a prime ideal of $R$ and $N$ be a submodule of $M$ with $Ann_R(N) \subseteq \mathfrak{p}$. Then $Ann_R(M) \subseteq \mathfrak{p}$. Thus by the proof of \cite[Lemma 2.10]{MR3091647}, 
$\mathfrak{p}=Ann_R((0:_M\mathfrak{p}))$. Now we have
$$
Ann_R((0:_N\mathfrak{p}))=Ann_R((0:_M\mathfrak{p})\cap N)=Ann_R((0:_M\mathfrak{p}+Ann_R(N)))
$$
$$
=Ann_R((0:_M\mathfrak{p}))=\mathfrak{p}.
$$

(b) This follows from part (a) and \cite[Proposition 12]{MR3934877}.

(c)
Let $\mathfrak{p}$ be a prime ideal of $R$ such that $Ann_R(N)\subseteq \mathfrak{p} \subseteq Ann_R(S)$. Then $S \subseteq (0:_N\mathfrak{p})$.  By \cite[Lemma 2.10]{MR3091647}, $(0:_N\mathfrak{p})$ is a second submodule of $M$. Thus $S=(0:_N\mathfrak{p})$. This implies that
$Ann_R(S)=\mathfrak{p}$ by part (a).

(d) Since $Ann_R(S)$ is a prime ideal of $R$, we have $S$ is a second submodule of $M$ by \cite[Theorem 196]{MR3934877}.
Now assume that $S_1$ be a miximal second submodule of $M$ such that $S\subseteq S_1$. Then by part (a), $Ann_R(S_1)$ is a prime ideal minimal over $Ann_R(M)$ and $Ann_R(S_1) \subseteq Ann_R(S)$. Thus as $Ann_R(S)$ is a prime ideal minimal over
$Ann_R(M)$, we have $Ann_R(S_1) = Ann_R(S)$. It follows that $S_1=S$, as needed.

(e) By part (c),  $Ann_R(S)$ is a minimal prime ideal of $R$ over $Ann_R(M)$.
As $M$ is a coreduced $R$-module, $R/Ann_R(M)$ is a reduced ring by \cite[Proposition 2.15]{MR3755273}.
Let  $a+Ann_R(M) \in Ann_R(S)/Ann_R(M)$.
Then by \cite[Proposition 1.2 (1)]{MR698302}, we have $Ann_{R/Ann_R(M)}(a+Ann_R(M)) \not \subseteq  Ann_R(S)/Ann_R(M)$.
\end{proof}

\begin{thm}\label{tt1.11}
Let $M$ be a faithful finitely generated coreduced comultiplication $R$-module. Then for each $a \in R$,  we have $V^s((0:_MAnn_R(a))=Max^s(M) \setminus V^s((0:_Ma))$.
\end{thm}
\begin{proof}
If $S\in V^s((0:_Ma))$, then we have $a \in Ann_R(S)$. Now by Proposition \ref{l0.1} (e), $Ann_R(a) \not\subseteq Ann_R(S)$ and so $S  \not\subseteq (0:_MAnn_R(a))$. Therefore, $V^s((0:_MAnn_R(a))\cap V^s((0:_Ma))=\emptyset$. On the other
hand, if $S\in Max^s(M) \setminus V^s((0:_Ma))$, then $aAnn_R(a)=0 \subseteq Ann_R(S)$ implies that $Ann_R(a) \subseteq Ann_R(S)$.
Hence, $S \in V^s((0:_MAnn_R(a))$ as needed.
\end{proof}

The intersection of all minimal prime ideals of $R$ containing an ideal $I$ of $R$ is denote by $\mathfrak{P}_I$.
\begin{thm}\label{t119.3}
Let $M$ be a faithful finitely generated comultiplication $R$-module. Then $(0:_M\mathfrak{P}_I)=\mathfrak{S}_{(0:_MI)}$ for each ideal $I$ of $R$.
\end{thm}
\begin{proof}
First we show that
$\mathfrak{S}_{(0:_MI)}\subseteq (0:_M\mathfrak{P}_I)$.
Let $I$ be an ideal of $R$ and $X$ be a maximal second submodule of $M$ such that $X \subseteq (0:_MI)$. Then $I \subseteq Ann_R(X)$. Since by Proposition \ref{l0.1} (c),  $Ann_R(X)$ is a minimal prime ideal of $R$, $\mathfrak{P}_I \subseteq Ann_R(X)$. Thus  $X=(0:_MAnn_R(X)) \subseteq (0:_M\mathfrak{P}_I)$. Hence,
$\mathfrak{S}_{(0:_MI)}\subseteq (0:_M\mathfrak{P}_I)$ for each ideal $I$ of $R$. For the converse, let $\mathfrak{P}_I=\cap \mathfrak{p}_i$, where $\mathfrak{p}_i \in Min(R), I \subseteq \mathfrak{p}_i$, and $Min(R)$ is the set of all minimal prime ideals of $R$. As $M$ is a faithful finitely generated comultiplication $R$-module, by using Proposition \ref{l0.1} (b),
$$
\sum_{(0:_M\mathfrak{p}_i)\not=0}(0:_M\mathfrak{p}_i)=\sum (0:_M\mathfrak{p}_i)=(0:_M\cap{\mathfrak{p}_i})=(0:_M\mathfrak{P}_I)\subseteq (0:_MI).
$$

This implies that $(0:_M\mathfrak{P}_I)\subseteq \mathfrak{S}_{(0:_MI)}$ since by Proposition \ref{l0.1} (d), $(0:_M\mathfrak{p}_i)\not=0$ is a maximal second submodule of $M$.
\end{proof}

\begin{rem}\label{r2.1}\cite[Remark 2.1]{MR2821719}
Let $N$ and $K$ be two submodules of an $R$-module $M$. To prove $N\subseteq K$, it is enough to show that if $L$ is a completely irreducible submodule of $M$ such that $K\subseteq L$, then $N\subseteq L$.
\end{rem}

\begin{thm}\label{t1.5}
Let $M$ be a finitely generated coreduced comultiplication $R$-module. Then we have the following.
\begin{itemize}
\item [(a)] If $M$ is a faithful $R$-module, then $V^s((0:_Ma))=V^s(Ann_R(a)M)$ for each  $a \in R$.
\item [(b)] $Ann_R(IJM)M=Ann_R(IM)M + Ann_R(JM)M$ for each ideals $I, J$ of $R$.
\end{itemize}
\end{thm}
\begin{proof}
(a)  Let $M$ be a faithful $R$-module and $a \in R$.
Since $Ann_R(a)M\subseteq (0:_Ma)$, we have $V^s(Ann_R(a)M))\subseteq V^s((0:_Ma))$. Now let
$S$ be a maximal second submodule of $M$ contained in $(0:_Ma)$. Then there exists
 $b \in Ann_R(a) \setminus Ann_R(S)$ by Proposition \ref{l0.1} (e). Then, for any completely irreducible submodule $L$ of $M$ with $ Ann_R(a)M\subseteq L$, we have $S\subseteq M=(L:_Mb)$.  This implies that $S=bS\subseteq L$. Hence, by Remark \ref{r2.1}, $S\subseteq Ann_R(a)M$, as needed.

(b) Let $I, J$ be ideals of $R$. As $Ann_R(IM) \subseteq Ann_R(IJM)$ and $Ann_R(JM) \subseteq Ann_R(IJM)$, we have
$$
Ann_R(IM)M + Ann_R(JM)M\subseteq Ann_R(IJM)M.
$$
Now let $L$ be a completely irreducible submodule of $M$ such that  $Ann_R(IM)M + Ann_R(JM)M\subseteq L$. Let $t \in Ann_R(IJM)$. Then $Jt \subseteq Ann_R(IM)$ and so $tJM \subseteq L$. Since $M$ is a comultiplication module, $(L:_Mt)=(0:_MA)$ for some ideal $A$ of $R$. Hence, $A \subseteq Ann_R(JM)$ and so $AM\subseteq L$. Thus $M=(L:_MA)$. It follows that $M=L+(0:_MA)$ since $M$ is a coreduced module. So, $M=L+(L:_Mt)$. Thus $tM\subseteq L$. Therefore, $Ann_R(IJM)M\subseteq L$. Now the proof follows from Remark \ref{r2.1}.
\end{proof}

\begin{prop}\label{p4.6}
Let $M$ be a finitely generated coreduced comultiplication $R$-module. Then the following are equivalent:
\begin{itemize}
\item [(a)] For submodules $K, H$ of $M$, $(K:_RM)=(H:_RM)$ and $N \subseteq K$ imply that $N \subseteq H$;
\item [(b)] For submodules $K, H$ of $M$, $(K:_RM)\subseteq (H:_RM)$ and $N \subseteq K$ imply that $N \subseteq H$.
\end{itemize}
\end{prop}
\begin{proof}
$(a)\Rightarrow (b)$
Let $K, H$ be submodules of $M$. As $M$ is a comultiplication $R$-module, there exist ideals $I$ and $J$ of $R$ such that $K=(0:_MI)$ and $H=(0:_MJ)$.
Assume that $(K:_RM)\subseteq (H:_RM)$ and  $N \subseteq K$. Then $Ann_R(IM)M \subseteq Ann_R(JM)M$. Hence
$Ann_R(IJM)M=Ann_R(JM)M$ by Theorem \ref{t1.5} (b). This implies that $Ann_R(IJM)=Ann_R(JM)$ and so $((0:_MIJ):_RM)=((0:_MJ):_RM)$. Now as $IJN=0$ we have $JN=0$ by part (a). Thus $N\subseteq (0:_MJ)=H$.

$(b)\Rightarrow (a)$
This is clear.
\end{proof}

\begin{cor}\label{cc4.6}
Let $M$ be a finitely generated coreduced comultiplication $R$-module. Then the following are equivalent:
\begin{itemize}
\item [(a)] For $a, b\in R$, $Ann_R(aM)=Ann_R(bM)$ and $aN=0$ imply that $bN=0$;
\item [(b)] For $a, b\in R$, $Ann_R(aM)\subseteq Ann_R(bM)$ and $aN=0$ imply that $bN=0$.
\end{itemize}
\end{cor}
\begin{proof}
This follows from Proposition \ref{p4.6} by setting $K=(0:_Ma)$ and $H=(0:_Mb)$.
\end{proof}

Let $S$ be a second submodule of an $R$-module $M$. Then (see \cite{MR3073398, MR2917107, MR3588217})
$$
I^M_{Ann_R(S)}(M)=\cap \{L \mid  L \\\ is \\\ a \\\ completely \\\
irreducible
\\\ submodule \\\ of \\\ M\\\ and
$$
$$
 (L:_RM) \not \subseteq Ann_R(S)\} .
$$

\begin{prop}\label{p0.1}
Let $S$ be a second submodule of an $R$-module $M$. Then $S \subseteq I^M_{Ann_R(S)}(M)$.
\end{prop}
\begin{proof}
Let $L$ be a completely irreducible submodule of $M$ such that  $(L:_RM) \not \subseteq Ann_R(S)$. Then there exists $t \in (L:_RM)$ such that $t \not \in Ann_R(S)$. Thus, as $S$ is second, $tS=S$. Now $t \in (L:_RM)$ implies that $S \subseteq L$. Therefore, $S \subseteq I^M_{Ann_R(S)}(M)$.
\end{proof}

For each prime ideal $\mathfrak{p} $ of $R$, set $nil \mathfrak{p}  = \cap \mathfrak{p}'$, where $\mathfrak{p}'$ ranges over all prime ideals of $R$ contained in $\mathfrak{p}$.

Let $\mathfrak{p} $ be a prime ideal of $R$. Then the set
$O_{\mathfrak{p}} = \{a \in R : Ann_R(a)\not \subseteq \mathfrak{p} \}$
is an ideal of $R$ contained in $\mathfrak{p}$ and $\sqrt{O_\mathfrak{p}} = nil \mathfrak{p}$, so if $\mathfrak{p} $ is a minimal
prime ideal of $R$, $\mathfrak{p} = \sqrt{O_\mathfrak{p} }$, in particular, $\mathfrak{p}= O\mathfrak{p}$, when $R$ is a reduced ring \cite{MR2839935}.
\begin{note}
Let $S$ be a second submodule of an $R$-module $M$.
We define $conil(S)=\sum S' $, where $S' $ ranges over all second submodules of $M$ such that $S\subseteq S' $.
\end{note}

For a submodule $N$ of an $R$-module $M$, the \emph{second radical} (or \emph{second socle}) of $N$ is defined  as the sum of all second submodules of $M$, contained in $N$, and it is denoted by $sec(N)$ (or $soc(N)$). In case $N$ does not contain any second submodule, the second radical of $N$ is defined to be $(0)$. $N \not =0$ is said to be a \emph{second radical submodule of $M$} if $sec(N)=N$ \cite{MR3085034, MR3073398}.

We set $\overline{I}=I+Ann_R(M)$ for each ideal $I$ of $R$ and $\overline{a}=a+Ann_R(M)$ for each $a \in R$.
\begin{thm}\label{t0.2}
Let $S$ be a second submodule of a finitely generated comultiplication $R$-module $M$. Then we have the following.
\begin{itemize}
\item [(a)]
$$
conil(S) = (0:_Mnil(\overline{Ann_R(S)}) =(0:_M\sqrt{O_{\overline{Ann_R(S)}}}).
$$
\item [(b)]
$$
sec(I^M_{Ann_R(S)}(M))=conil(S)=(0:_M\sqrt{O_{\overline{Ann_R(S)}}}).
$$
In particular, if $M$ is coreduced, then $sec(I^M_{Ann_R(S)}(M))=I^M_{Ann_R(S)}(M)$.
\end{itemize}

\end{thm}
\begin{proof}
(a) Since $S$ is second, we have $\overline{Ann_R(S)}$ is a prime ideal of $\overline{R}$. Thus $\sqrt{O_{\overline{Ann_R(S)}}} = nil \overline{Ann_R(S)}$ by \cite{MR2839935}. Hence, $ (0:_M(nil(\overline{Ann_R(S)}) =(0:_M\sqrt{O_{\overline{Ann_R(S)}}})$.
Now let $S'$ be a second submodule of $M$ with $S\subseteq S'$. Then $Ann_R(S') \subseteq Ann_R(S)$. Thus $nil(\overline{Ann_R(S)})\subseteq \overline{Ann_R(S')}$. By using \cite[Proposition 12]{MR3934877}, it follows that
$$
conil(S)=\sum S'=\sum (0:_M\overline{Ann_R(S')})=(0:_M\cap \overline{Ann_R(S')})
$$
$$
\subseteq  (0:_Mnil(\overline{Ann_R(S)}),
$$
where $S'$ ranges over all second submodules of $M$ such that $S\subseteq S'$.
On the other hand by Proposition \ref{l0.1} (b),
$$
(0:_Mnil(\overline{Ann_R(S)}) =(0:_M\cap {\mathfrak{p}})=\sum (0:_M {\mathfrak{p}})\subseteq conil(S),
$$
where $\mathfrak{p}$ ranges over all prime ideals of $R$ contained in $\overline{Ann_R(S)}$.

(b) Assume that $S'$ is a second submodule of $M$ such that $S\subseteq S'$. Then $Ann_R(S')\subseteq Ann_R(S)$. Thus for any completely irreducible submodule $L$ of $M$ with $(L:_RM) \not \subseteq Ann_R(S)$, we have $(L:_RM) \not \subseteq Ann_R(S')$.
Hence there exists $t \in (L:_RM)$ such that $t \not \in Ann_R(S')$. Thus, as $S'$ is second, $tS'=S'$. Now $t \in (L:_RM)$ implies that $S' \subseteq L$. It follows that $conil(S)\subseteq sec(I^M_{Ann_R(S)}(M))$.
Now let $\overline{a}\in nil(\overline{Ann_R(S)})$. Then for some positive integer $n$, $Ann_R(a^n) \not \subseteq Ann_R(S)$. This implies that $Ann_R(a^nM) \not \subseteq Ann_R(S)$. Therefore for each completely irreducible submodule $L$ of $M$, $((L:_Ma^n):_RM) =(L:_Ra^nM) \not \subseteq Ann_R(S)$. Since by \cite[Lemma 2.1]{MR3588217}, $(L:_Ra^nM)$ is a completely irreducible submodule of $M$, we have $I^M_{Ann_R(S)}(M)\subseteq (L:_Ra^nM)$ and so $a^nI^M_{Ann_R(S)}(M) \subseteq L$. Thus $a^nI^M_{Ann_R(S)}(M) \subseteq \cap L$, where $L$ ranges over all completely irreducible submodules of $M$. Thus $a^nI^M_{Ann_R(S)}(M)=0$. So $a \in \sqrt{Ann_R(I^M_{Ann_R(S)}(M))}$.
Thus by \cite[Proposition 2.1]{MR3091647},  we have
$$
sec(I^M_{Ann_R(S)}(M))=sec((0:_MAnn_R(I^M_{Ann_R(S)}(M))))\subseteq
$$
$$
(0:_M\sqrt{Ann_R(I^M_{Ann_R(S)}(M))})\subseteq (0:_M\sqrt{O_{\overline{Ann_R(S)}}}).
$$
This completes the proof of the first part by using part (a).
For the second part, let $M$ be a coreduced $R$-module. Then $R/Ann_R(M)$ is a
reduced ring by \cite[Proposition 2.15]{MR3755273}, and this implies that $\sqrt{O_{\overline{Ann_R(S)}}}=O_{\overline{Ann_R(S)}}$.  Now we have
$$
I^M_{Ann_R(S)}(M)\subseteq (0:_MO_{(\overline{Ann_R(S)})}))= (0:_M\sqrt{O_{(\overline{Ann_R(S)})}})=sec(I^M_{Ann_R(S)}(M)).
$$
This completes the proof because the reverse inclusion is clear.
\end{proof}

\begin{cor}\label{c00.3}
Let $S$ be a submodule of a finitely generated comultiplication $R$-module $M$. Then $S$ is a maximal second submodule of $M$ if and only if
$S=sec(I^M_{Ann_R(S)}(M))$. In particular, if $M$ is a coreduced $R$-module, then $S$ is a maximal second submodule of $M$
if and only if $S=I^M_{Ann_R(S)}(M)$.
\end{cor}
\begin{proof}
By \cite[Proposition 2.22]{MR3755273}, $sec(M) = M$. Now the result follows from Theorem \ref{t0.2}.
\end{proof}

\begin{note}
For a submodule $N$ of an $R$-module $M$, we define
$$
V^*(N) = \{S \in Max^s(M) : S \subseteq N \ and \ (N:_RM) \not \subseteq Ann_R(S)\}.
$$
Clearly, $V^*(N)\subseteq V^s(N)$.
\end{note}

Recall that a submodule $N$ of an $R$-module $M$ is said to be \emph {pure} if $IN=N \cap IM$ for every ideal $I$ of $R$ \cite{AF74}.
\begin{thm}\label{t0.3}
Let $L$ be a completely irreducible submodule of a comultiplication $R$-module $M$. Then $V^s(L) \subseteq V^*(L)$  if one of the following conditions hold.
\begin{itemize}
\item [(a)] $M$ is a faithful finitely generated coreduced $R$-module and $Ann_R(L)$ is a finitely generated ideal of $R$.
\item [(b)] $M$ is a finitely generated coreduced $R$-module and $M/S$ is a finitely cogenerated $R$-module for each maximal second submodule $S$ of $M$ with $S \subseteq L$.
\item [(c)] $L$ is a pure submodule of $M$.
\end{itemize}
 \end{thm}
\begin{proof}
(a) Let $S$ be a maximal second submodule of $M$ with $S \subseteq L$. Then $Ann_R(L) \subseteq Ann_R(S)$. By Proposition \ref{l0.1} (e), for each $a \in Ann_R(L)$, we have $Ann_R(a) \not \subseteq Ann_R(S)$. Now since $Ann_R(L)$ is a finitely generated ideal of $R$, $Ann_R(Ann_R(L)) \not \subseteq Ann_R(S)$. This implies that
$$
(L:_RM)=((0:_MAnn_R(L)):_RM)=(0:_RAnn_R(L)M)=
$$
$$
Ann_R(Ann_R(L)) \not \subseteq Ann_R(S).
$$
Therefore, $V^s(L) \subseteq V^*(L)$.

(b) Let $S$ be a maximal second submodule of $M$ with $S \subseteq L$. By Corollary \ref{c00.3},  $S=I^M_{Ann_R(S)}(M)$.  Now as $M/S$ is a finitely cogenerated $R$-module, there exist completely irreducible submodules $L_1, L_2, ..., L_n$ of $M$ with $(L_i:_RM) \not \subseteq Ann_R(S)$ and $S=\cap_{i=1}^nL_i$. Assume contrary that $(L:_RM) \subseteq Ann_R(S)$. Then $(S:_RM) \subseteq Ann_R(S)$ and so $(\cap_{i=1}^nL_i:_RM) \subseteq Ann_R(S)$. This implies that there exists $1\leq j\leq n$ such that $(L_j:_RM) \subseteq Ann_R(S)$, which is a contradiction.

(c) Let $S$ be a maximal second submodule of $M$ with $S \subseteq L$.  Assume contrary that $(L:_RM) \subseteq Ann_R(S)$. Then as $M$ is a comultiplication $R$-module $Ann_R(Ann_R(L)M)\subseteq Ann_R(S)$ and so $S\subseteq Ann_R(L)M$. It follows that $S\subseteq Ann_R(L)M \cap L=Ann_R(L)L=0$, which is a contradiction.
\end{proof}

\begin{prop}\label{lll0.3}
Let $M$ be an $R$-module such that $V^s(L) \subseteq V^*(L)$ for each completely irreducible submodule $L$ of $M$. If $K$ be a submodule of $M$ such that $M/K$ is a finitely cogenerated $R$-module and $S$ is a maximal second submodule of $M$ with $S \subseteq K$, then $(K:_RM) \not \subseteq Ann_R(S)$.
 \end{prop}
\begin{proof}
As $M/K$ is a finitely cogenerated $R$-module, there exist completely irreducible submodules $L_1, L_2, ..., L_n$ of $M$ such that $K=\cap_{i=1}^nL_i$. Now since $V^s(L_i) \subseteq V^*(L_i)$,  there exits $a_i \in (L_i:_RM) \setminus Ann_R(S)$ for each $1\leq i\leq n$. Set $a=a_1a_2...a_n$. Then we have $a \in (K:_RM)\setminus Ann_R(S)$ and so $(K:_RM) \not \subseteq Ann_R(S)$.
\end{proof}

\begin{cor}\label{cl0.3}
Let $M$ be a Noetherian coreduced comultiplication $R$-module. Then $V^*(K)= V^s(K)$ for each submodule $K$ of $M$.
 \end{cor}
\begin{proof}
By \cite[Corollary 2.11]{MR2782697}, $M$ is an Artinian $R$-module. Thus $M/K$ is finitely cogenerated for each submodule $K$ of $M$.
Now the result follows from Theorem \ref{t0.3} (b) and Proposition \ref{lll0.3}.
\end{proof}

\begin{prop}\label{pp1.5}
Let $M$ be a comultiplication $R$-module. Then we have the following.
\begin{itemize}
\item [(a)] If $R/Ann_R(M)$ is a reduced ring, then $M$ is a coreduced $R$-module.
\item [(b)] If $M$ is a faithful coreduced $R$-module, then $IM$ is a coreduced $R$-module for each ideal $I$ of $R$.
\end{itemize}
\end{prop}
\begin{proof}
(a) Let $r^2M\subseteq L$ for some $r \in R$ and a completely irreducible submodule $L$ of $M$. Then $r^2Ann_R(L)M=0$ and so for each $a \in Ann_R(L)$ we have  $r^2aM=0$. This implies that $(ra)^2\in Ann_R(M)$. Now as $R/Ann_R(M)$ is a reduced ring, we have $ra\in Ann_R(M)$. Hence  $rAnn_R(L)M=0$. It follows that $rM\subseteq L$ since $M$ is a comultiplication $R$-module. Thus by \cite[Theorem 2.13]{MR3755273},
$M$ is a coreduced $R$-module.

(b)
Let $I$ be an ideal of $R$. Since $M$ is a coreduced $R$-module, $R$ is a reduced ring by \cite[Proposition 2.15]{MR3755273}.
Thus $R/Ann_R(I)$ is a reduced ring by \cite[Lemma 2.6]{MR194880}. So $R/Ann_R(IM)$ is a reduced ring.  By \cite[Proposition 6]{MR3934877}, $IM$ is a comultiplication $R$-module. So,  by part (a), $IM$ is a coreduced $R$-module.
\end{proof}

Let $M$ be an $R$-module.
If $N$ is a submodule of $M$, then $sec(N)$ is the sum of the maximal second submodules of $N$ \cite[Proposition 2.4]{MR3091647}.
Clearly, $\mathfrak{S}_N\subseteq sec(N)$ for each submodule $N$ of $M$. In \cite[Proposition 2.22]{MR3755273}, it is shown that $M$ is a coreduced module if $sec(M) = M$ and the converse holds when $M$ is a finitely generated comultiplication $R$-module.
\begin{thm}\label{tt0.1}
Let $M$ be a finitely generated faithful coreduced comultiplication $R$-module. Then $\mathfrak{S}_{IM}= sec(IM)=IM$ for each ideal $I$ of $R$.
\end{thm}
\begin{proof}
Clearly $\mathfrak{S}_{IM}\subseteq sec(IM)$. For the reverse inclusion, let $S$ be a maximal second submodule of $IM$. It is enough to show that $S$ is a maximal second submodule of $M$.
By the proof of Proposition \ref{pp1.5} (b), $R/Ann_R(I)$ is a reduced ring.
By Proposition \ref{l0.1} (c), $Ann_R(S)$ is a minimal prime ideal over $Ann_R(IM)$. Since $M$ is faithful, we have $Ann_R(IM)=Ann_R(I)$. Thus $Ann_R(S)/Ann_R(I)$ is a minimal prime ideal of $R/Ann_R(I)$. Hence  by \cite[Proposition 2.10 (b)]{F402}, $Ann_R(S)$ is a minimal prime ideal of $R$. Thus $S$ is a maximal second submodule of $M$ by Proposition \ref{l0.1} (d) and we are through.  By \cite[Proposition 2.22]{MR3755273}, $sec(M)=M$. Hence, $Isec(M)=IM$. We show that $Isec(M) \subseteq sec(IM)$. So let $L$ be a completely irreducible submodule of $M$ with $sec(IM)\subseteq L$. Assume contrary that $Isec(M)\not \subseteq L$. Then there exists a maximal second submodule $S$ of $M$ such that $IS\not \subseteq L$. Since $S$ is second, $IS=S$ and so $S\not \subseteq L$. On the other
hand, $S=IS \subseteq IM$ implies that $S\subseteq sec(IM) \subseteq L$. This contradiction shows that $Isec(M) \subseteq L$. Thus by Remark \ref{r2.1}, $Isec(M) \subseteq sec(IM)$.
Therefore, $IM \subseteq sec(IM)$. Now, we are through since the reverse inclusion is clear.
\end{proof}

An $R$-module $M$ is said to be a \emph{multiplication module} if for every submodule $N$ of $M$ there exists an ideal $I$ of $R$ such that $N=IM$ \cite{Ba81}.
\begin{cor}\label{ccc1.13}
Let $M$ be a faithful finitely generated coreduced comultiplication and multiplication $R$-module. Then for each submodule $N$ of $M$, we have  $\mathfrak{S}_{N}= sec(N)=N$.
\end{cor}

\begin{ex}\label{e3.9}
Consider the $\Bbb Z_n$-module $M=\Bbb Z_n$, where $n$ is square free.
We know that $M$ is a faithful finitely generated coreduced comultiplication and multiplication $\Bbb Z_n$-module. Thus
for each submodule $N$ of $M$, we have  $\mathfrak{S}_{N}= sec(N)=N$.
\end{ex}

\begin{cor}\label{c1.13}
Let $M$ be a faithful finitely generated coreduced comultiplication $R$-module. Then for each $a \in R$, $Ann_R(aM)M=\mathfrak{S}_{(0:_Ma)}$.
\end{cor}
\begin{proof}
By Theorem \ref{tt0.1}, $IM=\mathfrak{S}_{IM}$ for each ideal $I$ of $R$. Now the result follows from Theorem \ref{t1.5} (a).
\end{proof}

The following example shows that the condition "$M$ is a finitely generated $R$-module" in Corollary \ref{c1.13} is necessary.
\begin{ex}\label{e1.13}
For each prime number $p$ the $\Bbb Z$-module $\Bbb Z_{p^\infty}$ is a faithful coreduced comultiplication $\Bbb Z$-module. But
the $\Bbb Z$-module $\Bbb Z_{p^\infty}$ is not finitely generated.
For each positive integer $n$,
$$
Ann_{\Bbb Z}(p^n\Bbb Z_{p^\infty})\Bbb Z_{p^\infty}=0\not =\langle1/p+ \Bbb Z\rangle=\mathfrak{S}_{(0:_{\Bbb Z_{p^\infty}}p^n)}.
$$
\end{ex}

\begin{thm}\label{t1.45}
Let $M$ be a finitely generated coreduced comultiplication $R$-module. Then for each submodule $K$ of $M$ we have the following.
\begin{itemize}
\item [(a)] $V^s((0:_M(K:_RM))=Max^s(M) \setminus V^*(K)$.
\item [(b)] $V^*(K)=V^*((K:_RM)M)$.
\end{itemize}
\end{thm}
\begin{proof}
(a)
If $S\in V^*(K)$, then $(K:_RM) \not\subseteq Ann_R(S)$. Thus $S \not\subseteq (0:_M(K:_RM)$. Therefore,  $V^s((0:_M(K:_RM))\cap V^*(K)=\emptyset$. On the other
hand, if $S\in Max^s(M) \setminus V^*(K)$, then $S\not \subseteq K$ or $(K:_RM)\subseteq Ann_R(S)$. If $(K:_RM)\subseteq Ann_R(S)$, then $S \subseteq (0:_M(K:_RM))$. Thus $S \in V^s((0:_M(K:_RM))$. If $S\not \subseteq K$,
then for any $b \in (K:_RM)$, we have $bS\subseteq K$. Since
$S\not \subseteq K$ and $S$ is second, $b\in Ann_R(S)$. This implies that $S \in V^s((0:_M(K:_RM))$.  Thus $V^s((0:_M(K:_RM))=Max^s(M) \setminus V^*(K)$.

(b) Clearly,  $V^*((K:_RM)M)\subseteq  V^*(K)$. Now let
$S$ be a maximal second submodule of $M$ contained in $K$ and $a \in (K:_RM) \setminus Ann_R(S)$. Thus for any completely irreducible submodule $L$ of $M$ with  $(K:_RM)M \subseteq L$, we have $0\not=aS\subseteq L$. It follows that $S \subseteq L$. Therefore, $S\subseteq (K:_RM)M$. As $(K:_RM)\not \subseteq Ann_R(S)$, we have  $((K:_RM)M:_RM)\not \subseteq Ann_R(S)$. Therefore,  $V^*(K)\subseteq V^*((K:_RM)M)$.
\end{proof}

\begin{cor}\label{c1.143}
Let $M$ be a Noetherian coreduced comultiplication $R$-module. Then for each submodule $K$ of $M$, $(K:_RM)M=\mathfrak{S}_{K}$.
\end{cor}
\begin{proof}
By Theorem \ref{tt0.1}, $(K:_RM)M=\mathfrak{S}_{(K:_RM)M}$. By \cite[Corollary 2.11]{MR2782697}, $M$ is an Artinian $R$-module.  Thus $M/K$ is a finitely cogenerated.
Now the result follows from Corollary \ref{cl0.3} and Theorem \ref{t1.45} (b).
\end{proof}

\begin{thm}\label{t99.6}
Let $N$ be a submodule of a Noetherian coreduced comultiplication $R$-module $M$. Then for submodules $K$ and $H$ of $M$ the following are equivalent:
\begin{itemize}
\item [(a)]  $\mathfrak{S}_K=\mathfrak{S}_H$ and $N \subseteq  K$ imply that $N  \subseteq H$;
\item [(b)] $V^s(K)=V^s(H)$ and $N \subseteq K$ imply that $N  \subseteq H$;
\item [(c)]  $(K:_RM)=(H:_RM)$ and $N  \subseteq  K$ imply that $N  \subseteq H$;
\item [(d)] For submodule $K$ of $M$, $N \subseteq K$ implies that $N \subseteq (K:_RM)M$.
\end{itemize}
\end{thm}
\begin{proof}
First note that by Corollary \ref{cl0.3}, $V^s(K)=V^*(K)$ for each submodule $K$ of $M$.

$(a)\Rightarrow (b)$
Let  $V^s(K)=V^s(H)$ and $N \subseteq  K$. Then  $\mathfrak{S}_K=\mathfrak{S}_H$. Thus by part (a), $N \subseteq  H$.

$(b)\Rightarrow (c)$
Let $(K:_RM)=(H:_RM)$ and $N  \subseteq  K$. Then $(K:_RM)M=(H:_RM)M$.  So, by Theorem \ref{t1.45} (b), $V^s(K)=V^s(H)$. Hence $N \subseteq  H$ by part (b).

$(c)\Rightarrow (a)$
Let  $\mathfrak{S}_{K}=\mathfrak{S}_{H}$ and $N \subseteq  K$. Then $(K:_RM)M=(H:_RM)M$ by Corollary \ref{c1.143}. Thus $$
(K:_RM)=((K:_RM)M:_RM)=((H:_RM)M:_RM)=(H:_RM).
$$
So, $N \subseteq  H$  by part (c).

$(c)\Rightarrow (d)$
This follows from the fact that $(K:_RM)=((K:_RM)M:_RM)$.

$(d)\Rightarrow (c)$
Let for submodules $K$ and $H$ of $M$, $(K:_RM)=(H:_RM)$ and $N \subseteq K$. Then $(K:_RM)M=(H:_RM)M$. By part (d), $N\subseteq (K:_RM)M$. Thus $N\subseteq (H:_RM)M\subseteq H$.
\end{proof}

\end{document}